\documentclass[12pt]{article}
\usepackage{amsmath, amscd, amssymb, latexsym, epsfig, color, amsthm, authblk, subfig, caption, mathtools,hyperref,appendix}
\numberwithin{equation}{section}

\newtheorem{theorem}{Theorem}[section]
\newtheorem{proposition}[theorem]{Proposition}

\newtheorem{corollary}[theorem]{Corollary}

\newtheorem{lemma}[theorem]{Lemma}
\newtheorem{property}[theorem]{Property}
\newtheorem{problem}[theorem]{Problem}

\theoremstyle{definition}
\newtheorem{definition}[theorem]{Definition}
\newtheorem{example}[theorem]{Example}
\newtheorem{remark}[theorem]{Remark}

\DeclareMathOperator{\skel}{Skel}
\DeclareMathOperator{\lk}{\mathrm{lk}}
\DeclareMathOperator{\Aut}{\mathrm{Aut}}
\DeclareMathOperator{\st}{\mathrm{st}}
\DeclareMathOperator{\Sp}{\mathbb{S}}

\DeclareMathOperator{\conv}{\mathrm{conv}}

\newcommand{\R}{{\mathbb R}}

\newcommand{\Z}{{\mathbb Z}}

\begin{document}

\title{Centrally symmetric and balanced triangulations of $\Sp^2\times \Sp^{d-3}$ with few vertices}
\author[1]{Alexander Wang \thanks{alxwang@umich.edu}}
\author[2]{Hailun Zheng\thanks{hailunz@umich.edu}}
\affil[1,2]{University of Michigan}

\maketitle

\begin{abstract}
    A small triangulation of the sphere product can be found in lower dimensions by computer search and is known in few other cases: Klee and Novik constructed a centrally symmetric triangulation of $\Sp^i\times \Sp^{d-i-1}$ with $2d+2$ vertices for all $d\geq 3$ and $1\leq i\leq d-2$; they also proposed a balanced triangulation of $\Sp^1\times \Sp^{d-2}$ with $3d$ or $3d+2$ vertices. In this paper, we provide another centrally symmetric $(2d+2)$-vertex triangulation of $\Sp^2\times \Sp^{d-3}$. We also construct the first balanced triangulation of $\Sp^2\times \Sp^{d-3}$ with $4d$ vertices, using a sphere decomposition inspired by handle theory.
\end{abstract}

\section{Introduction}
The minimal triangulations of manifolds form an important research object in combinatorial and computational topology. What is the minimal number of vertices required to triangulate a given manifold? How do we construct a vertex-minimal triangulation and is this triangulation unique?

In this paper, we focus on the triangulation of sphere products. From a result of Brehm and K\"uhnel \cite{Brehm-Kuhnel}, it is known that a combinatorial triangulation of $\Sp^i\times\Sp^{d-i-1}$ has at least $2d-i+2$ vertices, for $1 \leq i \leq (d-1)/2$. In 1986, K\"uhnel \cite{Kuhnel} constructed a triangulation of $\Sp^1 \times \Sp^{d-2}$ with $2d+1$ vertices for odd $d$. Later, two groups of researchers, Bagchi and Datta \cite{Bagchi-Datta} as well as Chestnut, Sapir and Swartz \cite{Chestnut-Sapir-Swartz}, found in 2008 that K\"uhnel's construction is indeed the unique minimal triangulation for odd $d$. For even $d$, they showed that a minimal triangulation requires $2d+2$ vertices and is not unique.

Minimal triangulations of other sphere products are less well-understood. The best general result is from \cite{Klee-Novik-cs manifolds}, where a centrally symmetric triangulation of $\Sp^i\times \Sp^{d-i-1}$ with $2d+2$ vertices is constructed as a subcomplex of the boundary of the $(d+1)$-cross-polytope. In addition, a minimal triangulation of $\Sp^2\times \Sp^{d-3}$ for $d\leq 6$ as well as a minimal triangulation of $\Sp^3\times \Sp^3$ are found by the computer program BISTELLAR \cite{Lutz}. In this paper, we give an alternative centrally symmetric $(2d+2)$-vertex triangulation of $\Sp^2\times \Sp^{d-3}$ for all $d\geq 5$. The construction is based on finding two shellable $(d-1)$-balls in a $(d-1)$-sphere whose intersection triangulates $\Sp^1\times \mathbb{D}^{d-2}$, where $\mathbb{D}^{d-2}$ is the $(d-2)$-dimensional disk. We also describe an inductive method to construct triangulations of other sphere products in higher dimensions, see Section 3.2.

In recent years, balanced triangulated manifolds have caught much attention. A $(d-1)$-dimensional simplicial complex is \emph{balanced} provided that its graph is $d$-colorable. Many important classes of complexes arise as balanced complexes, such as barycentric subdivisions of regular CW complexes and Coxeter complexes. As taking barycentric subdivisions of a complex would generate a lot of new vertices, one would ask if there is a more efficient way to construct the balanced triangulated manifold from a non-balanced one.

In much of the same spirit as K\"uhnel's construction, Klee and Novik \cite{Klee-Novik-balanced LBT} provided a balanced triangulation of $\Sp^1\times \Sp^{d-2}$ with $3d$ vertices for odd $d$ and with $3d+2$ vertices otherwise. Furthermore, Zheng \cite{Zheng} showed that the number of vertices for a minimal triangulation is indeed $3d$ for odd $d$ and $3d+2$ otherwise. However, as of yet, no balanced triangulations of $\Sp^i\times \Sp^{d-i-1}$ for $2\leq i\leq d-3$ exist in literature. In this paper, we construct the first balanced triangulation of $\Sp^2\times \Sp^{d-3}$ with $4d$ vertices. The construction uses a sphere decomposition inspired by handle theory. Recently, Izmestiev, Klee, and Novik \cite{Izmestiev-Klee-Novik} proved that any two balanced PL homeomorphic closed combinatorial manifolds can be connected using a sequence of cross-flips. In particular, given a balanced triangulated manifold, this allows us to computationally search for a minimal balanced triangulation. However, the computation complexity grows very fast as the dimension and the number of vertices increase, see \cite[Theorem 2.4]{Kubitzke-Venturello}. So far, we have been unable to find a smaller balanced triangulation of $\Sp^2\times \Sp^2$ from our construction.

The paper is structured as follows. In Section 2, we review the basics of simplicial complexes, balanced triangulations, and other relevant definitions. In Section 3, we present our centrally symmetric $(2d+2)$-vertex triangulation of $\Sp^2\times \Sp^{d-3}$. In Section 4, the balanced triangulation of $\Sp^2\times \Sp^{d-3}$ with $4d$ vertices is constructed, followed by a discussion of its properties.

\section{Preliminaries}
A \textit{simplicial complex} $\Delta$ with vertex set $V (= V(\Delta)) $ is a collection of subsets
$\sigma\subseteq V$, called \textit{faces}, that is closed under inclusion, such that for every $v \in V$, $\{v\} \in \Delta$. For $\sigma\in \Delta$, let $\dim\sigma:=|\sigma|-1$ and define the \textit{dimension} of $\Delta$, $\dim \Delta$, as the maximum dimension of the faces of $\Delta$.  A face $\sigma \in \Delta$ is said to be a \textit{facet} provided that it is a face which is maximal with respect to inclusion.  We say that a simplicial complex $\Delta$ is \textit{pure} if all of its facets have the same dimension. If $\Delta$ is $(d-1)$-dimensional and $-1 \leq i \leq d-1$, then the \textit{$f$-number} $f_i = f_i(\Delta)$ denotes the number of $i$-dimensional faces of $\Delta$. The \textit{star} and \textit{link} of a face $\sigma$ in $\Delta$ is defined as follows:

\[
\st(\sigma, \Delta):= \{\tau \in\Delta: \sigma\cup\tau\in\Delta \}, \quad \lk(\sigma,\Delta):=\{\tau \in \st(\sigma,\Delta): \tau \cap \sigma = \emptyset\}.
\]
When the context is clear, we may simply denote the star and link of $\sigma$ as $\st(\sigma)$ and $\lk(\sigma)$, respectively. The \emph{cone} over the simplicial complex $\Delta$ with apex $v$ is denoted as $\Delta *\{v\}$. We also define the \textit{restriction} of $\Delta$ to a vertex set $W$ as $\Delta[W]:=\{\sigma\in \Delta:\sigma\subseteq W\}$. A subcomplex $\Omega \subset \Delta$ is said to be \textit{induced} provided that $\Omega = \Delta[U]$ for some $U \subset V(\Delta)$. If $\Delta$ and $\Gamma$ are pure simplicial complexes of the same dimension, the \emph{complement} of $\Gamma$ in $\Delta$, denoted as $\overline{\Delta\backslash \Gamma}$, is the subcomplex of $\Delta$ generated by all facets of $\Delta$ not in $\Gamma$. The \textit{$i$-skeleton} of a simplicial complex $\Delta$ is the subcomplex containing all faces of $\Delta$ which have dimension at most $i$.  In particular, the $1$-skeleton of $\Delta$ is the graph of $\Delta$.

Denote by $\sigma^d$ the $d$-simplex. A \emph{combinatorial $(d-1)$-sphere} (respectively, a \emph{combinatorial $(d-1)$-ball}) is a simplicial complex PL homeomorphic to $\partial \sigma^d$ (respectively, $\sigma^{d-1}$).  A closed \emph{combinatorial $(d-1)$-manifold} $\Delta$ (with or without boundary) is a connected simplicial complex such that the link of every non-empty face $F$ is a $(d-|F|-1)$-dimensional combinatorial ball or sphere; in the former case we say $F$ is a \emph{boundary face}. The \emph{boundary complex} of $\Delta$, denoted by $\partial \Delta$, is the subcomplex of $\Delta$ that consists of all boundary faces of $\Delta$. The combinatorial manifolds have the following nice property, see \cite[Theorems 12.6 and 14.4]{Alexander}.
\begin{lemma}\label{lm: union of comb manifolds}
	Let $\Delta_1$ and $\Delta_2$ be combinatorial $(d-1)$-manifolds with boundary. Then
	\begin{itemize}
	\item The boundary complexes $\partial\Delta_1$ and $\partial \Delta_2$ are combinatorial $(d-2)$-manifolds.
	\item If $\Delta_1\cap \Delta_2=\partial\Delta_1\cap \partial \Delta_2$ is a combinatorial $(d-2)$-manifold, then $\Delta_1\cup \Delta_2$ is a combinatorial $(d-1)$-manifold. If furthermore $\Delta_2$ and $\Delta_1\cap \Delta_2$ are combinatorial balls, then $\Delta_1\cup \Delta_2$ and $\Delta_1$ are PL homeomorphic manifolds.
	\end{itemize} 
\end{lemma}

For each simplicial complex $\Delta$ there is an associated topological space $\|\Delta\|$. We say $\Delta$ is \emph{simply connected} (resp. \emph{path connected}) if $\|\Delta\|$ is simply connected (resp. path connected). A simplicial complex $\Delta$ is called a \textit{simplicial manifold} if $\|\Delta\|$ is homeomorphic to a manifold. The boundary complex of a simplicial $d$-ball is a simplicial $(d-1)$-sphere. It is well-known that while simplicial manifolds of dimension $\leq 3$ are combinatorial, in general a simplicial manifold needs not be combinatorial. 

In the following sections we will study simplicial complexes with additional nice structures. A $(d-1)$-dimensional simplicial complex $\Delta$ is called \textit{balanced} if the graph of $\Delta$ is $d$-colorable; that is, there exists a coloring map $k: V\to \{1,2,\cdots,d\}$ such that $k(x) \neq k(y)$ for all edges $\{x,y\}\in\Delta$. A simplicial complex is \emph{centrally symmetric} or \emph{cs} if it is endowed with a free involution $\alpha: V(\Delta)\to V(\Delta)$ that induces a free involution on the set of all non-empty faces. We say $F$ and $-F:=\alpha(F)$ are \emph{antipodal faces}. In general, if $\Gamma$ and $-\Gamma:=\alpha(\Gamma)$ are subcomplexes of $\Delta$, then we say that they form \emph{antipodal complexes} in $\Delta$.

Let $C_d^*:=\conv\{\pm e_1, \dots, \pm e_d\}$ be the $d$-cross-polytope, where $e_1, \dots, e_d$ form the standard basis in $\R^d$. One important class of cs complexes is given by the boundary complex of $C_d^*$. In fact $\partial C_d^*$ is the cs and balanced vertex-minimal triangulation of the $(d-1)$-sphere. In Section 3, we will label the vertex set of $\partial C^*_d$ as $\{x_1,\dots, x_d,y_1,\dots, y_d\}$ such that $x_i, y_i$ form a pair of antipodal vertices for all $i$. Every facet of $\partial C^*_d$ can be written in the form $\{u_1,u_2,\dots ,u_d\}$, where each $u_i \in \{x_i,y_i\}$. We say a facet has a \emph{switch} at position $i$ if $u_i$ and $u_{i+1}$ have different $xy$ labels. 

Let $B(i,d)$ be the pure subcomplex of $\partial C^*_d$ generated by all facets with at most $i$ switches. For example, $B(0,d)$ consists of two disjoint facets $\{x_1,\dots, x_d\}$ and $\{y_1,\dots, y_d\}$. Denote by $\mathcal{D}_d$ the dihedral group of order $2d$. The following lemma is essentially Theorem 1.2 in \cite{Klee-Novik-cs manifolds}.
\begin{lemma}\label{lemma: B(i,d) properties}
	For $0\leq i<d-1$, the complex $B(i,d)$ satisfies the following properties:
	\begin{enumerate}
		\item $B(i,d)$ contains the entire $i$-skeleton of $\partial C_d^*$ as a subcomplex.
		\item The boundary of $B(i,d)$ triangulates $\Sp^i\times \Sp^{d-i-2}$.
		\item $B(i,d)$ is a balanced cs combinatorial manifold whose integral (co)homology groups coincide with those of $\Sp^i$. Also, $\|B(0,d)\|\cong \mathbb{D}^{d-1}\times \Sp^0$ and $\|B(1,d)\|\cong \mathbb{D}^{d-2}\times \Sp^1$.
		\item The complement of $B(i,d)$ in $\partial C^*_d$ is simplicially isomorphic to $B(d-i-2,d)$.
		\item B(i,d) admits a vertex-transitive action of $\mathbb{Z}_2\times \mathcal{D}_d$ if $i$ is even and of $\mathcal{D}_{2d}$ if $i$ is odd.
	\end{enumerate}
\end{lemma}
Finally, we define shellability.
\begin{definition}
	Let $\Delta$ be a pure $d$-dimensional simplicial complex. A \emph{shelling} of $\Delta$ is a linear ordering of the facets $F_1,F_2,\dots, F_s$ such that $F_k\cap (\cup_{i=1}^{k-1} F_i)$ is a pure $(d-1)$-dimensional complex for all $2\leq k\leq s$, and $\Delta$ is called \emph{shellable} if it has a shelling.
\end{definition}
Equivalently, $\Delta$ is called shellable if for each $j\geq 0$ there exists a face $r(F_j )\subseteq F_j$ such that $F_j\backslash \cup_{i<j}F_i=[r(F_j), F_j]$, where $[r(F_j), F_j]=\{G: r(F_i)\subseteq G\subseteq F_j\}$. The face $r(F_j)$ is called the \emph{restriction} of $F_j$. Note that not every combinatorial ball and combinatorial sphere are shellable, see, for example, \cite{Hachimori-Ziegler}.

\section{The cs triangulations of the sphere products}
It is known that for $i\leq j$, a minimal triangulation of $\Sp^i\times \Sp^j$ requires at least $i+2j+4$ vertices \cite{Brehm-Kuhnel}. Such triangulations are constructed by Lutz in lower dimensional cases but not known in general. 
In this section, we aim at finding another triangulation of $\Sp^2\times \Sp^{d-3}$ with $2d+2$ vertices for $d\geq 5$. The following theorem is Theorem 7 in \cite{Kreck}. Throughout, all the homology groups in the paper are computed with coefficients in $\mathbb{Z}$. We write $\tilde{H}_*(\Delta)$ to denote the reduced homology of $\Delta$ with coefficients in $\mathbb{Z}$, and $\beta_i(\Delta)$ denotes the rank of $\tilde{H}_i(\Delta)$. 

\begin{theorem}\label{thm: homology of product of spheres}
	Let $M$ be a simply connected codimension-1 submanifold of $\Sp^{d}$, where
	$d \geq 5$. If $M$ has the integral homology of $\Sp^i \times \Sp^{d-i-1}$ and $1 < i \leq \frac{d-1}{2}$, then $M$ is homeomorphic
	to $\Sp^i \times \Sp^{d-i-1}$.
\end{theorem} 

\begin{proposition}\label{prop: criteria}
	Fix $d\geq 5$ and $1\leq i\leq \frac{d-1}{2}$. Let $D_1$ and $D_2$ be two combinatorial $d$-balls such that
	\begin{enumerate}
	    \item $\partial (D_1\cup D_2)$ is combinatorial $(d-1)$-manifold contained in a combinatorial $d$-sphere.
	    \item $D_1\cap D_2=\partial D_1\cap \partial D_2$ is a path-connected combinatorial $(d-1)$-manifold (with boundary) that has the same homology as $\Sp^{i-1}$.
	    \item $\partial (D_1\cap D_2)$ is a combinatorial $(d-2)$-manifold having the same homology as $
\Sp^{i-1}\times \Sp^{d-i-1}$.
	\end{enumerate}
	 Then $\partial(D_1\cup D_2)$ is combinatorial triangulation of $\Sp^i\times \Sp^{d-i-1}$.
\end{proposition}   
\begin{proof}
	First note that $D_1\cup D_2$ is the union of two combinatorial $d$-balls that intersect along the combinatorial $(d-1)$-manifold $D_1\cap D_2=\partial D_1\cap \partial D_2$. Hence by Lemma \ref{lm: union of comb manifolds}, $D_1\cup D_2$ is a combinatorial $d$-manifold, and $\partial (D_1\cup D_2)$ is a combinatorial $(d-1)$-manifold.
	
    Since $D_1\cap D_2=\partial D_1 \cap \partial D_2$, we have that the intersection of $\overline{\partial D_1\backslash \partial D_2}$ and $D_1\cap D_2$ is exactly $\partial (D_1\cap D_2)$, while the union of $\overline{\partial D_1\backslash\partial D_2}$ and $D_1\cap D_2$ is $\partial D_1$. Applying the Mayer-Vietoris sequence on $(\overline{\partial D_1\backslash\partial D_2}, D_1\cap D_2,\partial D_1)$,  we obtain 
    \begin{equation}\label{sequence1}
    	\dots \to\tilde{H}_{j+1}(\partial D_1)\to \tilde{H}_j(\partial(D_1\cap D_2)) \xrightarrow{(\phi_1^j, f)} \tilde{H}_j(\overline{\partial D_1\backslash\partial D_2}) \oplus \tilde{H}_j(D_1\cap D_2) \to  \tilde{H}_j(\partial D_1) \to \cdots.
    \end{equation}
    Since $\partial D_1$, $D_1\cap D_2$ and $\partial (D_1\cap D_2)$ have the same homology as $\Sp^{d-1}$, $\Sp^{i-1}$ and $\Sp^{i-1}\times \Sp^{d-i-1}$ respectively, it follows that $\overline{\partial D_1\backslash\partial D_2}$ has the same homology as $\Sp^{d-i-1}$. 
    
    Note that the intersection \[\big(\overline{\partial D_1\backslash \partial D_2}\big)\cap\big(\overline{\partial D_2\backslash \partial D_1}\big)=\partial (\partial D_1\cap \partial D_2)=\partial (D_1\cap D_2),\]
    \[\mathrm{and} \quad \big(\overline{\partial D_1\backslash \partial D_2}\big)\cup\big(\overline{\partial D_2\backslash \partial D_1}\big)=\partial (D_1\cup D_2).\] We apply the Mayer-Vietoris sequence on the triple $(\overline{\partial D_1\backslash \partial D_2}, \overline{\partial D_2\backslash \partial D_1}, \partial(D_1\cup D_2))$. First assume that $d-i-1\neq i$. Since $\partial(D_1\cap D_2)$ has the same homology as $\Sp^{i-1}\times \Sp^{d-i-1}$,
    \begin{equation*}
    \begin{split}
     0\to \tilde{H}_{d-i}(\partial(D_1\cup D_2))\to\tilde{H}_{d-i-1}(\partial (D_1\cap D_2))&\xrightarrow{\phi^{d-i-1}} \tilde{H}_{d-i-1}(\overline{\partial D_1\backslash \partial D_2})\oplus \tilde{H}_{d-i-1}(\overline{\partial D_2\backslash \partial D_1})\\
    &\to \tilde{H}_{d-i-1}(\partial(D_1\cup D_2))\to 0,
    \end{split}   
    \end{equation*}
    Note that map $\phi^{d-i-1}$ is given by $a\mapsto (\phi_1^{d-i-1}(a),\phi^{d-i-1}_2(a))$, where
    $\tilde{H}_{d-i-1}(\partial (D_1\cap D_2))\xrightarrow{\phi_1^{d-i-1}} \tilde{H}_{d-i-1}(\overline{\partial D_1\backslash \partial D_2})$
    is the same map as in the sequence (\ref{sequence1}) above, and $\phi^{d-i-1}_2$ is a similar map $\tilde{H}_{d-i-1}(\partial (D_1\cap D_2))\xrightarrow{\phi_2^{d-i-1}} \tilde{H}_{d-i-1}(\overline{\partial D_2\backslash \partial D_1})$. Since $\phi_1^{d-i-1}, \phi_2^{d-i-1}$ are isomorphisms, it follows that $\tilde{H}_{d-i}(\partial (D_1\cup D_2))=0$ and $\tilde{H}_{d-i-1}(\partial (D_1\cup D_2))=\mathbb{Z}$. Also
    \[0\to \tilde{H}_{j}(\partial(D_1\cup D_2))\to \tilde{H}_{j-1} (\partial (D_1\cap D_2))\to 0 \quad \mathrm{for} \;\; j< d-i-1\;\; \mathrm{or}\;\; j>d-i.\]
    So $\partial (D_1\cup D_2)$ has the same homology as $\Sp^i\times \Sp^{d-i-1}$. The case for $d-i-1=i$ is similar; we have
    \[0\to \tilde{H}_{i+1}(\partial (D_1\cup D_2))\to \tilde{H}_{i}(\partial (D_1\cap D_2))\xrightarrow{(\phi_1^i, \phi_2^i)} \tilde{H}_{i}(\overline{\partial D_1\backslash \partial D_2})\oplus \tilde{H}_{i}(\overline{\partial D_2\backslash \partial D_1})\]
    \[ \to \tilde{H}_{i}(\partial(D_1\cup D_2))\to \tilde{H}_{i-1}(\partial (D_1\cap D_2))\to 0,\]
    As before, $\phi_1^i, \phi_2^i$ are isomorphisms (induced from injections from $\partial (D_1\cap D_2)$ to $\overline{\partial D_1\backslash \partial D_2}$ and $\overline{\partial D_2\backslash \partial D_1}$ respectively). So it reduces to 
    $$0\to \Z\xrightarrow{(\phi_1^i, \phi_2^i)} \Z\oplus \Z \to  \tilde{H}_{i}(\partial (D_1\cup D_2)) \to \Z\to 0.$$ Hence $\tilde{H}_{i}(\partial(D_1\cup D_2))=\Z\oplus\Z$. The same argument as above shows that $\partial(D_1\cup D_2)$ has trivial homology elsewhere.
    
    Finally, the complex $D_1\cup D_2$ is simply connected, since the union of two simply connected open subsets $\mathrm{int} D_1, \mathrm{int}D_2$ with path-connected intersection $D_1\cap D_2$ is simply connected. We conclude from Theorem \ref{thm: homology of product of spheres} that $\partial (D_1\cup D_2)$ triangulates $\Sp^i\times \Sp^{d-i-1}$.
\end{proof}

The above proposition provides us with a general method of constructing a triangulation of $\Sp^i\times \Sp^{d-i-1}$. In this paper we will mainly use the following variation of Proposition \ref{prop: criteria}.
\begin{corollary}\label{cor: inductive method}
	Let $d\geq 5$, $i\geq 2$ and let $B_1$ and $B_2$ be two combinatorial $(d-1)$-balls in a combinatorial $(d-1)$-sphere $\Gamma$ such that
	\begin{enumerate}
		\item $B_1\cap B_2$ is a path-connected combinatorial $(d-1)$-manifold having the same homology as $\Sp^{i-1}$.
		\item $\partial (B_1\cap B_2)$ is a combinatorial $(d-2)$-manifold having the same homology as $\Sp^{i-1}\times \Sp^{d-i-1}$.
	\end{enumerate}
	Let $D_1=B_1*\{u\}$ and $D_2=B_2*\{v\}$, where $u,v$ are the new vertices. Then $\Delta=\partial (D_1\cup D_2)$ triangulates $\Sp^i\times \Sp^{d-i-1}$. Furthermore, $\Delta$ is cs if $B_1$ and $B_2$ are antipodal.
\end{corollary}
\begin{proof}
	Since $B_1, B_2$ are combinatorial $(d-1)$-balls, $D_1, D_2$ are combinatorial $d$-balls. Also $D_1\cap D_2=B_1\cap B_2$ is a combinatorial $(d-1)$-manifold. By Lemma \ref{lm: union of comb manifolds} the complex $D_1\cup D_2$ is a combinatorial $d$-manifold whose boundary is in $\Gamma$. Then the first claim immediately follows from Proposition \ref{prop: criteria}. The second claim is obvious.
\end{proof}

\subsection{A triangulation of $\Sp^2\times \Sp^{d-3}$}
In the following we use the convention that $x_{d+i}:=x_i$ and $y_{d+i}:=y_i$. Let $\tau$ be a face of $\partial C_d^*$ and let $\kappa(\tau)$ count the number of $y$ labels in $\tau$. Define $\Gamma_j$ to be the subcomplex whose facets are those $\tau$ in $\partial C_d^*$ that have at most 2 switches and with $\kappa(\tau)=j$. Hence for $j=0$, the complex $\Gamma_0$ consists a single facet $\{x_1,\dots, x_d\}$ and for $1\leq j\leq d-1$, the complex $\Gamma_j$ consists of $d$ facets  $\tau_j^k=\{x_1,\dots,x_d\}\backslash\{x_{k}, \dots ,x_{k+j-1}\}\cup \{y_{k},\dots, y_{k+j-1}\}$ for $1\leq k\leq d$.

\begin{lemma}\label{lm: shellable ball}
	The complex $\cup_{k=0}^{i}\Gamma_k$ is a shellable $(d-1)$-ball for all $0\leq i\leq \left\lceil \frac{d+1}{2}\right\rceil$.
\end{lemma}
\begin{proof}
	We prove by induction on $i$. The complex $\Gamma_0$ consists of one facet $\{x_1,\dots,x_d\}$ and $\Gamma_1$ contains every adjacent facet of $\{x_1,\dots, x_d\}$ in $\partial C_d^*$, hence both $\Gamma_0$ and $\Gamma_0\cup \Gamma_1$ are shellable balls. Now assume that $\Delta:=\cup_{k=0}^{i-1}\Gamma_k$ is a shellable $(d-1)$-ball. Note that for any $k$ and $2\leq i\leq \left\lceil \frac{d+1}{2}\right\rceil$, \[\tau_i^k \cap (\Delta\cup \cup_{j=1}^{k-1}\tau_i^{j})=\tau_i^k \cap \Delta=\big(\tau_i^k\backslash\{y_k\}\big)\cup \big(\tau_i^k\backslash \{y_{k+i-1}\}\big).\] In other words, the restriction face $r(\tau_i^k)$ is the edge $\{y_k, y_{k+i-
		1}\}$ as long as $i\leq \left\lceil \frac{d+1}{2}\right\rceil$. Hence by the inductive hypothesis and induction on $k$, $\cup_{k=0}^{i}\Gamma_k$ is simplicial $(d-1)$-ball that has a shelling order $\tau_0:=\{x_1,\dots, x_d\},\tau_1^1,\dots, \tau_1^d,\dots ,\tau_i^1,\dots,\tau_i^d$.
\end{proof}

We propose the candidates $D_1,D_2\subseteq \partial C_{d+1}^*$ that satisfy the conditions in Proposition \ref{prop: criteria}.
\begin{definition}\label{def: D_1,D_2}
	For $d\geq 3$, define two combinatorial $d$-balls $D_1,D_2$ as a subcomplex of $\partial C_{d+1}^*$ on vertex set $\{x_1, y_1,\dots, x_{d+1},y_{d+1}\}$ as follows: 
	\begin{enumerate}
		\item If $d=2m+1$ is odd, define $D_1=(\cup_{k=0}^{m+1}\Gamma_k)*\{x_{d+1}\}$ and $D_2=(\cup_{k=m}^{d}\Gamma_k)*\{y_{d+1}\}$. In particular, $D_1\cap D_2=\Gamma_{m}\cup \Gamma_{m+1}$ is cs.
		\item If $d=2m$ is even, let $\gamma:=\cup_{i=1}^{m} \tau_{m-1}^i$ be a subcomplex of $\Gamma_{m-1}$. By the definition, $\tau_j^k$ and $\tau_{d-j}^{k+j}$ are antipodal facets in $\partial C_d^*$ for any $k,j$. So $-\gamma=\cup_{i=m}^{d-1}\tau_{m+1}^i\subseteq \Gamma_{m+1}$. In this case we let \[D_1=\big((\cup_{k=0}^{m}\Gamma_k) \cup (-\gamma)\big)*\{x_{d+1}\}, \quad D_2=\big((\cup_{k=m}^{d}\Gamma_k)\cup \gamma\big)*\{y_{d+1}\}.\] 
		In particular, $D_1\cap D_2=\Gamma_m\cup \gamma\cup (-\gamma)$ is cs.
	\end{enumerate} 
\end{definition}
	

Next we show that $\partial(D_1\cap D_2)\cong \Sp^1\times \Sp^{d-3}$. Given two facets $F_1, F_2\in \partial C_d^*$, let $d(F_1,F_2)$ be the distance from $F_1$ to $F_2$ in the facet-ridge graph of $\partial C_d^*$. 
\begin{lemma}\label{lm: criteria of sphere disc product}
	Let $\Delta$ be a combinatorial $(d-1)$-manifold in $\partial C_d^*$ whose facet-ridge graph is a $2d$-cycle. Enumerate its facets as $\sigma_1,\sigma_2,\dots \sigma_{2d}$ such that $\sigma_i, \sigma_{i+1}$ are adjacent for $1\leq i\leq 2d$. If $\sigma_i=-\sigma_{d+i}$ for all $1\leq i\leq d$, then $\Delta$ triangulates $\Sp^1\times \mathbb{D}^{d-2}$.
\end{lemma}
\begin{proof}
	Let $\sigma_1=\{u_1,\dots, u_d\}$. By the assumption, $\sigma_{d+1}=\{-u_1,\dots, -u_d\}$. Since $d(\sigma_1, \sigma_{d+1})=d$ in $\partial C_d^*$, the sequence $\sigma_1,\sigma_2,\dots, \sigma_{d+1}$ gives the shortest path from $\sigma_1$ to $\sigma_{d+1}$. So it follows that there is an ordering of the vertices, say $(u_1,\dots, u_d)$, such that $\sigma_{i+1}=\sigma_i\backslash\{u_{i}\}\cup \{-u_i\}$. Together with the fact that $\sigma_i=-\sigma_{i+d}$ for all $i$, we have that $\Delta\cong B(1,d)$. Hence as $B(1,d)$, $\Delta$ also triangulates $\Sp^1\times \mathbb{D}^{d-2}$.
\end{proof}
\begin{lemma}\label{lm: D_1, D_2}
	The complex $D_1\cap D_2$ constructed above triangulates $\Sp^1\times \mathbb{D}^{d-2}$.
\end{lemma}
\begin{proof}
	For odd $d$ and $m=\frac{d-1}{2}$, we enumerate the facets of $D_1\cap D_2=\Gamma_m\cup\Gamma_{m+1}$ as $(\sigma_1,\dots, \sigma_{2d}):=$
	\[(\tau_{m}^1,\tau_{m+1}^1, \tau_m^2, \tau_{m+1}^2,\dots, \tau_m^d, \tau_{m+1}^d).\] 
	Each $\sigma_i$ has exactly two adjacent facets $\sigma_{i-1}, \sigma_{i+1}$, and so the facet-ridge graph of  $D_1\cup D_2$ is a $2d$-cycle. Furthermore, since $\tau_m^j=-\tau_{m+1}^{j+m}$ by the definition, we have $\sigma_i=-\sigma_{d+i}$ for all $1\leq i\leq d$. 
	
	For even $d$ and $m=\frac{d}{2}$, we enumerate the facets of $D_1\cap D_2=\Gamma\cup\gamma\cup(-\gamma)$ as $(\sigma_1,\dots, \sigma_{2d}):=$
	\[(\tau_{m-1}^1, \tau_m^1, \tau_{m-1}^2, \tau_m^2, \dots, \tau_{m-1}^m, \tau_m^m, \tau_{m+1}^m, \tau_m^{m+1}, \tau_{m+1}^{m+1}, \dots, \tau_m^{d-1}, \tau_{m+1}^{d-1},\tau_m^{d}).\]
	As before, each $\sigma_i$ has exactly two adjacent facets $\sigma_{i-1}, \sigma_{i+1}$, and the facet-ridge graph of  $D_1\cap D_2$ is a $2d$-cycle. Also $\tau_m^i$ and $\tau_m^{i+m}$,  $\tau_{m-1}^i$ and $\tau_{m+1}^{i+m-1}$ are antipodal by the definition, so again $\sigma_i=-\sigma_{d+i}$ for all $i$. Our claim then follows from Lemma \ref{lm: criteria of sphere disc product}.
\end{proof}
\begin{theorem}\label{cor: intersection of D_1 D_2 is a manifold}
	The complex $\partial (D_1\cup D_2)$ gives a cs triangulation of $\Sp^2\times \Sp^{d-3}$ for $d\geq 5$.
\end{theorem}
\begin{proof}
	The complex $\partial(D_1\cup D_2)$ is cs follows from the definition and the fact that $D_1$ and $D_2$ form antipodal complexes in $\partial C_{d+1}^*$. By Lemma \ref{lm: shellable ball}, the complexes $D_1, D_2$ are shellable balls in the $d$-sphere $\partial C_{d+1}^*$, no matter whether $d$ is odd or even. Also by Lemma \ref{lm: D_1, D_2}, $D_1\cap D_2$ is path-connected and triangulates $\Sp^1\times \mathbb{D}^{d-2}$, and $\partial (D_1\cap D_2)$ triangulates $\Sp^1\times \Sp^{d-3}$. We conclude from Corollary \ref{cor: inductive method} that $\partial (D_1\cup D_2)$ triangulates $\Sp^2\times \Sp^{d-3}$. Finally, $\partial (D_1\cup D_2)$ is cs since $D_1$ and $D_2$ form antipodal complexes in $\partial C_{d+1}^*$.
\end{proof}

\begin{proposition}\label{property: construction 1} The complex $\partial (D_1\cup D_2)$ has the following properties.
	\begin{enumerate}
	\item It has $2d+2$ vertices.
	\item It contains the 2-skeleton of $\partial C_{d+1}^*$.
	\item $\partial (D_1\cup D_2)$ admits vertex-transitive actions by the group $\mathbb{Z}_2\times \mathcal{D}_d$ if $d$ is odd, and by $\mathbb{Z}_2$ if $d$ is even.
\end{enumerate}
\end{proposition}
\begin{proof}
	Part (1) is obvious by the construction. For part (2), note that $D_1\cap D_2\cong B(1,d)$ has the same graph as $C_d^*$ by Lemma \ref{lemma: B(i,d) properties} and $D_1, D_2$ are cones with apex $x_{d+1}, y_{d+1}$ respectively. So it follows that every 2-face in $\partial C_{d+1}^*$ that contains either $x_{d+1}$ or $y_{d+1}$ is in $\partial(D_1\cup D_2)$. Furthermore, $\partial (D_1\cup D_2)$ also contains $\skel_2(\partial C_d^*)$ by the definition. Hence $\partial (D_1\cup D_2)\supseteq \skel_2(\partial C_{d+1}^*)$. 
	
	For part (3), first notice that $\partial (D_1\cup D_2)$ is centrally symmetric, and so it admits a group action $D$ that maps $x_j$ to $y_j$, and $y_j$ to $x_j$, for $0\leq j\leq d$. Next consider a group action that fixes $x_{d+1}$ and $y_{d+1}$. Note that the facets in $\partial(D_1\cup D_2)$ which don't contain either $x_{d+1}$ or $y_{d+1}$ are those in the complement of $D_1\cap D_2$ in $\partial C_d^*$. Hence in the case when $d$ is odd, the vertex-transitive group actions are given by the following permutations $R, S$ as in \cite{Klee-Novik-cs manifolds}:
	\begin{itemize}
		\item $R$ fixes $x_{d+1}, y_{d+1}$, and maps $x_j, y_j$ to $x_{d-j+1}, y_{d-j+1}$ respectively. 
		\item $S$ fixes $x_{d+1}, y_{d+1}$, and maps $x_j, y_j$ to $x_{j+1}, y_{j+1}$ (modulo $d$) respectively.
	\end{itemize}
   However, for even $d$, every action on the complex must send $\lk(x_{d+1}, D_1)$, $\lk(y_{d+1}, D_2)$ and the complement of $D_1\cap D_2$ in $\partial C_d^*$ to themselves. Since these complexes are less symmetric than those in the case where $d$ is odd, none of the permutations of $R, S$ satisfy this condition. This proves the claim.
\end{proof}
\begin{remark}
By comparing the vertex-transitive actions on $\partial(D_1\cup D_2)$ and $\partial B(2,d+1)$ given in Proposition \ref{property: construction 1} and \cite{Klee-Novik-cs manifolds} see also Lemma \ref{lemma: B(i,d) properties} part 5), it follows that $\partial(D_1\cup D_2)$ and $\partial B(2,d+1)$ are \emph{distinct} centrally symmetric triangulations of $\Sp^2\times \Sp^{d-3}$.
\end{remark}


\subsection{Constructing sphere products by induction}
The goal of this section is to construct a triangulation of $\Sp^{i+1}\times \Sp^{d-i-1}$ from a given triangulation  of $\Sp^{i}\times \Sp^{d-i-1}$ for $i \leq (d-1)/2$. Recall that in \cite[Lemma 3.3]{Klee-Novik-cs manifolds} it is proved that
\begin{lemma}\label{lm: B(i,d) links}
	The intersection of the links of $x_d$ and $y_d$ in $B(i,d)=\st(x_d, B(i,d))\cup \st(y_d, B(i,d))$ is $B(i-1, d-1)$.
\end{lemma}
\noindent Below we describe the inductive method:

\noindent {\textbf{Step 1:}} Let $1\leq i<\frac{d}{2}$. Let $A_1, A_2, B_1, B_2$ be combinatorial $(d-1)$-balls in a combinatorial $(d-1)$-sphere such that $B_1\cap B_2=A_1\cup A_2$. Furthermore let $$D(i-1, d)=(A_1*\{u\})\cup (A_2* \{v\}),\quad D(i,d)=(B_1*\{u\})\cup (B_2*\{v\}).$$ In addition we have that 
	\begin{enumerate}
		\item $A_1\cap A_2$ and $B_1\cap B_2$ are combinatorial $(d-1)$-manifolds with boundary having the same homology as $\Sp^{i-2}$ and $\Sp^{i-1}$ respectively.
		\item $\partial (A_1\cap A_2)$ and $\partial (B_1\cap B_2)$ are combinatorial $(d-2)$-manifolds having the same homology as $\Sp^{i-2}\times \Sp^{d-i}$ and $\Sp^{i-1}\times \Sp^{d-i-1}$ respectively.
	\end{enumerate}

\noindent {\textbf{Output:}} By Corollary \ref{cor: inductive method}, $D(i-1,d)$ and $D(i,d)$ are combinatorial $d$-manifolds that have the same homology as $\Sp^{i-1}$ and $\Sp^{i}$ respectively. We also generate two combinatorial $(d-1)$-manifolds $\partial D(i-1, d)$ and $\partial D(i,d)$ that triangulate $\Sp^{i-1}\times \Sp^{d-i}$ and $\Sp^{i}\times \Sp^{d-i-1}$ respectively.

\noindent {\textbf{Step 2:}} Let $C_1=(B_1*\{u\})\cup (A_2*\{v\})$ and $C_2=(A_1*\{u\})\cup (B_2*\{v\})$ and furthermore, $D(i, d+1)=(C_1*\{u'\})\cup (C_2*\{v'\}).$ 

\noindent {\textbf{Claim}:} $C_1$ and $C_2$ are combinatorial $d$-balls and $D(i,d+1)$ is a combinatorial $(d+1)$-manifold.

\begin{proof}
	The intersection of two combinatorial $d$-balls $B_1*\{u\}$ and $A_2*\{v\}$ is $B_1\cap A_2=A_2$, which is a combinatorial $(d-1)$-ball contained in their boundaries. Hence $C_1$ is also a combinatorial $d$-ball. Similarly, $C_2$ is also a combinatorial $d$-ball. Since $B_1\cap B_2=A_1\cup A_2$, we have that
	\[C_1\cap C_2=(A_1*\{u\})\cup (A_1\cap A_2)\cup (B_1\cap B_2)\cup (A_2*\{v\})=D(i-1,d),\]
	which is a combinatorial $d$-manifold.
	Hence by Lemma \ref{lm: union of comb manifolds} indeed $D(i,d+1)$ is a combinatorial $(d+1)$-manifold.
\end{proof}

\noindent {\textbf{Output:}} By Corollary \ref{cor: inductive method}, $D(i,d+1)$ is a combinatorial $(d+1)$-manifold having the same homology as $\Sp^i$, furthermore $\partial D(i,d+1)$ triangulates $\Sp^i \times \Sp^{d-i}$.
\begin{remark}
	If $(A_1, A_2)$ and $(B_1, B_2)$ form antipodal subcomplexes in a centrally symmetric sphere, then all $D(j,d)$ (for $j=i-1, i, i+1$) and their boundary complexes are centrally symmetric.
\end{remark}
\begin{example}
	We discuss a trivial application in the case $i=1$. Let $B_1$ and $B_2$ be combinatorial $(d-1)$-balls such that $B_1\cap B_2$ consists of two disjoint facets $A_1, A_2$ (which has $d$ vertices respectively). Furthermore $B_1\cup B_2$ triangulates $\Sp^1\times \mathbb{D}^{d-2}$. The above method generate $D(1, d+1)$ whose boundary triangulates $\Sp^1\times \Sp^{d-2}$. Any triangulation of $\Sp^1\times \Sp^{d-2}$ constructed by this method has at least $2d+2$ vertices.
\end{example}
\begin{example}
	A nontrivial application is that we could obtain $B(i,d)$ from the above inductive method. Indeed, 
	 $$B(j, d)=\Big(\lk(x_{d}, B(j, d))* \{x_{d}\}\Big)  \cup \Big( \lk(y_{d}, B(j, d))*\{y_{d}\}\Big)  \quad \mathrm{for} \;\;j=i-1,i, \;\;\mathrm{and}$$ 
    $$\lk(x_{d}, B(i-1, d))\cup \lk(y_{d}, B(i-1, d))=\lk(x_{d}, B(i, d))\cap \lk(y_{d}, B(i, d))=B(i-1,d-1).$$

    Since $\lk(x_{d+1}, B(i, d+1))=\st(x_d, B(i,d))\cup \st(y_d, B(i-1,d))$, the inductive method reconstructs the complex $B(i, d+1)$ whose boundary triangulates $\Sp^i\times \Sp^{d-i-1}$.
\end{example}
\begin{problem}
	Find a triangulation of $\Sp^i\times \Sp^j$ with less than $2i+2j+4$ vertices for $j\geq i\geq 2$.
\end{problem}

\section{A balanced triangulation of $\Sp^2\times \Sp^{d-3}$}
In this section, we present our main construction for a balanced triangulation of $\Sp^2 \times \Sp^{d-3}$.  The geometric intuition of our construction comes from handle theory. For $d\geq 3$ the sphere $\Sp^{d-1}$ admits the following decomposition:

\[ \Sp^{d-1} = \partial\mathbb{D}^d = (\partial \mathbb{D}^2 \times \mathbb{D}^{d-2}) \cup (\mathbb{D}^2 \times \partial \mathbb{D}^{d-2}) = (\Sp^1 \times \mathbb{D}^{d-2}) \cup (\mathbb{D}^2 \times \Sp^{d-3}). \]

Let $S$ be a triangulated $(d-1)$-sphere that has the decomposition $S=B_1\cup_{\partial B_1=\partial B_2} B_2$ such that $\|B_1\|\cong \Sp^1\times \mathbb{D}^{d-2}$, $\|B_2\|\cong \mathbb{D}^2\times \Sp^{d-3}$, and $\|\partial B_1\| = \|\partial B_2\|\cong \Sp^1 \times \Sp^{d-3}$. From $S$ we can form a triangulation of $\Sp^2\times \Sp^{d-2}$ in the following way: take two copies of $B_2$ and denote them as $B_2$ and $B_2'$. If $\partial B_2$ is an \emph{induced} subcomplex in $B_2$, then we glue $B_2$ and $B_2'$ along their boundaries. The resulting complex is homeomorphic to $\Sp^2\times \Sp^{d-3}$. However, if $\partial B_2$ is not an induced subcomplex of $B_2$, then usually we cannot glue $B_2$ and $B_2'$ by identifying their boundaries directly to obtain a triangulated manifold. Instead, one needs to construct a simplicial complex $N$ such that $\|N\|\cong \|\partial B_2\|\times \mathbb{D}^1$ and $\partial N=\partial B_2\cup\partial B_2'$. (Geometrically, $\|N\|$ serves as a tubular neighborhood of both $\|\partial B_2\|$ and $\|\partial B_2'\|$.) Then the complex $B_2\cup N\cup B_2'$ is a triangulation of $\Sp^2\times \Sp^{d-3}$.

Our approach of constructing a balanced triangulation of $\Sp^2\times \Sp^{d-3}$ is by finding suitable balanced candidates of $B_2$ and $N$ as described above. We begin with defining a variation of the usual connected sum.
\begin{definition}\label{Definition: connected sum}
	Consider $(\Gamma_1, \sigma_1)$ and $(\Gamma_2, \sigma_2)$, where $\Gamma_1$ and $\Gamma_2$ are boundary complexes of two $d$-cross-polytopes defined on disjoint vertex sets, and $\sigma_1, \sigma_2$ are facets of $\Gamma_1, \Gamma_2$ respectively. For $i=1,2$, let $\kappa_i$ be the coloring map on $\Gamma_i$. If $e_i$ is an edge in $\Gamma_i$ but not in $\pm \sigma_i$ and $\kappa_1(e_1)=\kappa_2(e_2)$, then the $\Diamond$-connected sum $(\Gamma_1\#\Gamma_2,\sigma_1\#\sigma_2)$ is obtained by deleting $e_i$ from $\Gamma_i$, and gluing $\overline{\Gamma_1\backslash \st(e_1, \Gamma_1)}$ with $\overline{\Gamma_2\backslash \st(e_2,\Gamma_2)}$ along their boundaries in such a way that  $\st(e_1)[V(\sigma_1)]$ is identified with $\st(e_2)[V(\sigma_2)]$, and $\st(e_1)[V(-\sigma_1)]$ is identified with $\st(e_2)[V(-\sigma_2)]$. The new coloring map is given by $\kappa: V\to [d]$, $\kappa(v):=\kappa_i(v)$ if $v\in \Gamma_i$ for $i=1,2$.
\end{definition}
\begin{figure}[h]
	\centering
	\subfloat[$\Gamma_1$ and $\Gamma_2$; here $\sigma_1=\{y_1,y_2,y_3\}$ and $\sigma_2=\{x_1,Y_2,Y_3\}$]{\includegraphics{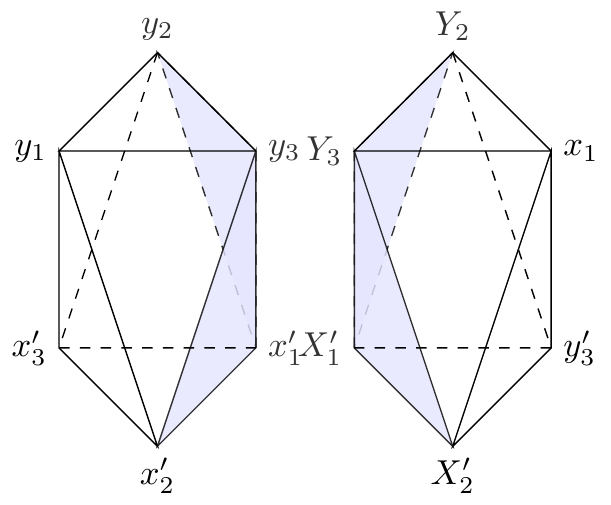}}
	\hspace{0.8cm}
	\subfloat[$(\Gamma_1\#\Gamma_2,\sigma_1\#\sigma_2)$]{\includegraphics[scale=1.2]{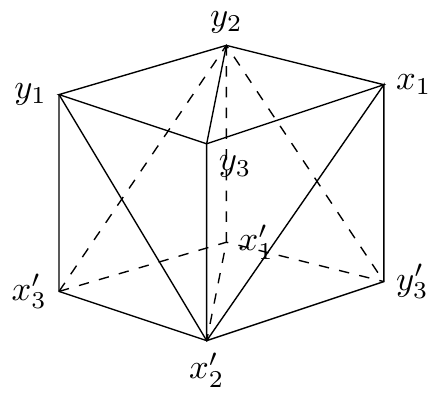}}
	\caption{The $\Diamond$-connected sum $(\Gamma_1\#\Gamma_2,\sigma_1\#\sigma_2)$. First we delete the edge $\{y_3,x_1'\}$ in $\Gamma_1$ and $\{Y_3, X_1'\}$ in $\Gamma_2$, then glue $\Gamma_1$ and $\Gamma_2$ along the 4-cycles $(y_3,x_2',x_1',y_2)$ and $(Y_3, X_2', X_1', Y_2)$.}
\end{figure}
The following properties of the $\Diamond$-connected sum justify the notation $(\Gamma_1\#\Gamma_2,\sigma_1\#\sigma_2)$ in the definition. 
\begin{lemma} Let $\Gamma_1$ and $\Gamma_2$ be the boundary complexes of two $d$-cross-polytopes with pairs of antipodal facets $\pm \sigma_1, \pm \sigma_2$ respectively. Then $(\Gamma_1\#\Gamma_2,\sigma_1\#\sigma_2)$ satisfies the following properties:
	\begin{enumerate}
		\item The complex is a balanced triangulation of $\Sp^{d-1}$.
		\item The restriction of $(\Gamma_1\#\Gamma_2,\sigma_1\#\sigma_2)$ to $V(\sigma_1)\cup V(\sigma_2)$ is the usual connected sum of simplices $\sigma_1\# \sigma_2$.
		\item The link of every edge $e=\{x_i,y_j\}$ in $(\Gamma_1\#\Gamma_2,\sigma_1\#\sigma_2)$ is the boundary complex of a $(d-2)$-cross-polytope.
	\end{enumerate}	
\end{lemma}
\begin{proof}
	 Assume that $\sigma_1=\{x_1,\dots,x_d\}, \;-\sigma_1=\{y_1,\dots,y_d\}$, and $\sigma_2=\{x_{d+1},\dots,x_{2d}\}, -\sigma_2=\{y_{d+1},\dots,y_{2d}\}$. Part 1 is clear from the construction. For part 2, assume that $e_1=\{x_i, y_j\}, e_2=\{x_k, y_l\}$ are the edges in $\Gamma_1, \Gamma_2$ deleted to form $(\Gamma_1\#\Gamma_2,\sigma_1\#\sigma_2)$. The link $\lk(e_1,\Gamma_1)$ is the boundary of a $(d-2)$-cross-polytope containing the antipodal facets $\sigma_1\backslash\{x_i,x_j\}$ and $(-\sigma_1)\backslash\{y_i,y_j\}$. Similarly, the link $\lk(e_2,\Gamma_2)$ has the antipodal facets $\sigma_2\backslash\{x_k,x_l\}$ and $(-\sigma_2)\backslash\{y_k,y_l\}$. Hence the restriction of $(\Gamma_1\#\Gamma_2,\sigma_1\#\sigma_2)$ to $V(\sigma_1)\cup V(\sigma_2)$ is obtained by taking the union of $\sigma_1$ and $\sigma_2$ and identifying $\sigma_1\backslash\{x_j\}$ with $\sigma_2\backslash\{x_l\}$. In this manner, we get the connected sum $\sigma_1\#\sigma_2$.
	
	For part 3, let $\Delta$ denote the boundary complex of the edge star on which $\Gamma_1$ and $\Gamma_2$ are glued together. If $e\notin \Delta$, there is nothing to prove. Otherwise, assume without loss of generality that $e=\{x_1,y_2\}$ and the edge $e'=\{x_1, y_3\}$ is deleted from $\Gamma_1$ to form $(\Gamma_1\#\Gamma_2,\sigma_1\#\sigma_2)$. Then, $\lk(e, \st(e',\Gamma_1))=\{y_3\}*\Sigma$, where $\Sigma$ is the boundary of the cross-polytope on vertices $\{x_4,y_4,\dots,x_d,y_d\}$. Hence, by construction, the link of $e$ in $(\Gamma_1\#\Gamma_2,\sigma_1\#\sigma_2)$ must be the suspension of $\Sigma$, i.e., the boundary of a $(d-2)$-cross-polytope.
\end{proof}

The above properties ensure that it is possible to take the  $\Diamond$-connected sum inductively. To form $(\Gamma_1\#\dots \#\Gamma_k, \sigma_1\#\dots \#\sigma_k)$ from $(\Gamma_1\#\dots \#\Gamma_{k-1}, \sigma_1\#\dots \#\sigma_{k-1})$ and $(\Gamma_k,\sigma_k)$, we take an edge $e_1\in (\Gamma_1\#\dots \#\Gamma_{k-1}, \sigma_1\#\dots \#\sigma_{k-1})$ but not in $\pm (\sigma_1\#\dots \sigma_{k-1})$, then take an edge $e_2\in \Gamma_k\backslash \pm\sigma_k$ so that $e_1$ and $e_2$ have the same colors, and then take $\diamond$-connected sum as in Definition \ref{Definition: connected sum}.

Recall that if $\Gamma$ is a pure simplicial complex and furthermore there exist two facets $F$ and $F'$ on $\Gamma$ and a map $\phi: F\to F'$ such that for every $v \in F$, $v$ and $\phi(v)$ do not have a common neighbor, then we can remove $F, F'$ and identify $\partial F$ with $\partial F'$ to obtain a new complex $\Gamma^\phi$. This is called a \emph{handle addition}. Similarly, assume that there are two edges $e_1$ and $e_2$ of the same color in $(\Gamma_1\#\dots \#\Gamma_k, A)$ but not in $\pm A$, where $A:=\sigma_1\#\sigma_2\dots\#\sigma_k$. Note that the boundary of $\st(e_i)$ is cross-polytopal with antipodal facets $\st(e_i)[V(A)]$ and $\st(e_i)[V(-A)]$ for $i=1,2$. If the identification maps 
\[\phi: \st(e_1)[V(A)] \to \st(e_2)[V(A)] \quad \mathrm{and}\quad \phi': \st(e_1)[V(-A)] \to \st(e_2)[V(-A)]\] 
are well-defined, then the maps $\phi$ and $\phi'$ naturally extend to a map \[\bar{\phi}: \st(e_1)\to \st(e_2),\] if for every $v\in \st(e_1)$ the vertices $v$ and $\phi(v)$(or $\phi'(v)$) do not have a common neighbor. In this way we obtain a balanced simplicial complex $((\Gamma_1\#\Gamma_2\dots \#\Gamma_k)^{\bar{\phi}}, A^\phi)$ by removing $e_1,e_2$ and identifying $\lk(e_1)$ with $\bar{\phi}(\lk(e_1))=\lk(e_2)$. We call this the $\Diamond$-handle addition. Note that as long as the handle addition is well-defined, \[f_0((\Gamma_1\#\Gamma_2\dots \#\Gamma_k)^{\bar{\phi}})=2f_0(A^\phi)=2k.\]
\subsection{Main construction}
We are now ready to construct a balanced triangulation of $\Sp^2\times \Sp^{d-3}$ with $4d$ vertices. We will write $\Gamma_1\#\Gamma_2$ to denote the $\Diamond$-connected sum if $\sigma_1$ and $\sigma_2$ are clear from the context. Also, to simplify notation, we will sometimes write $x_1\dots x_m$ to denote the face $\{x_1,\dots,x_m\}$.

\begin{example}\label{main construction}
	Let $d\geq 3$. We use the convention that $x_{d+i}=x_i$ and $y_{d+i}=y_i$. Take two $d$-cross-polytopes $P$ and $P'$. The vertex sets of $P$ and $P'$ are $\{x_1,\dots,x_d,y_1,\dots, y_d\}$ and $\{x'_1,\dots,x'_d,y'_1,\dots, y'_d\}$ respectively. We let $\sigma_i = x_1\dots x_iy_{i+1}\dots y_d$ and $\sigma_{d+i}= y_1\dots y_ix_{i+1}\dots x_d$ for $1 \leq i \leq d$. Then the complex $\Delta_1$ generated by the facets $\sigma_1,\dots, \sigma_{2d}$ is exactly $B(1,d)$. We further partition the boundary of $P$ as $\partial P=\Delta_1 \cup_{\partial\Delta_1} \Delta_2$. By Lemma \ref{lemma: B(i,d) properties}, $\Delta_2\cong B(d-3,d)$ and $\|\Delta_1\cap \Delta_2\|\cong \Sp^1\times \Sp^{d-3}$.
	
	Next, define a simplicial map $f: \partial P\to \partial P'$ induced by the following bijection on the vertex sets:
	\[x_i \mapsto x'_{i+1}, \;y_i \mapsto y'_{i+1} \;\; \mathrm{for}\; 1\leq i\leq d-1;\; x_d \mapsto y'_1, \;y_d \mapsto x'_1. \]
	
	By Lemma \ref{lemma: B(i,d) properties}, the complex $\Delta_1$ admits a vertex-transitive action by the dihedral group $\mathcal{D}_{2d}$ of order $4d$, where a generator is given by the map we have chosen (see the proof of Theorem 1.2 in \cite{Klee-Novik-cs manifolds}). Hence $f$ is a simplicial isomorphism and $f(\Delta_1) \cong B(1,d)$. For each $i$, there is a unique $d$-cross-polytope $\Gamma_i$ containing $\sigma_i$ and $f(\sigma_i)$ as antipodal facets. Next, we check that we can take the $\Diamond$-connected sum of $\Gamma_1\#\dots \#\Gamma_{2d}$ inductively. Without loss of generality, assume that $1 \leq i \leq d$; otherwise, we can relabel by switching $x$ and $y$. Note that for $i\leq d-2$, \[\sigma_i = x_1x_2\dots x_iy_{i+1}y_{i+2}\dots y_d, \;\;\sigma_{i+1} = x_1x_2\dots x_{i+1}y_{i+2}y_{i+3}\dots y_d,\] \[\mathrm{and}\;\; f(\sigma_i) = x'_2x'_3\dots x'_{i+1}y'_{i+2}y'_{i+3}\dots y'_dx'_1, \;\; f(\sigma_{i+1}) = x'_2x'_3\dots x'_{i+2}y'_{i+3}y'_{i+4}\dots y'_dx'_1.\] 
	Hence, $\sigma_i \cap \sigma_{i+1} = x_1x_2\dots x_iy_{i+2}\dots y_d$ and $f(\sigma_i) \cap f(\sigma_{i+1}) = x'_2x'_3\dots x'_{i+1}y'_{i+3}\dots y'_dx'_1$.  The missing indices are $i+1$ and $i+2$ respectively, so we let $e_i=x'_{i+1}y_{i+2}$. It follows that $\Gamma_i \cap \Gamma_{i+1} = \st(e_i,\Gamma_i) =\st(e_i,\Gamma_{i+1})$ and hence the $\Diamond
	$-connected sum is well-defined.  Similarly, $\Gamma_{d-1}\cap \Gamma_{d}=\st(\{x'_d,x_1\}, \Gamma_d )$ and $\Gamma_d\cap \Gamma_{d+1}=\st(\{y_1',x_2\}, \Gamma_d)$. Inductively, we form a complex $\Gamma=((\Gamma_1\#\Gamma_2\dots\#\Gamma_{2d})^{\bar{\phi}}, \Delta_1)$ which contains $\Delta_1$ and $f(\Delta_1)$ as subcomplexes.
	
	We partition $\Gamma$ as $\Gamma=\Delta_1\cup f(\Delta_1)\cup N$, so that $N \cap \Delta_1=\partial \Delta_1$ and $N \cap f(\Delta_1)=\partial f(\Delta_1)$. $N$ is then the tubular neighborhood that we would like to construct. Finally, let $\Sigma=\Delta_2\cup_{\partial \Delta_1} N\cup _{\partial f(\Delta_1)} f(\Delta_2)$. (This is well defined as by Lemma \ref{lemma: B(i,d) properties}, $\partial \Delta_1 \cong \partial \Delta_2$.)
\end{example}

\begin{figure}[h]
	\centering
	\subfloat[$\Delta_1$]{\includegraphics{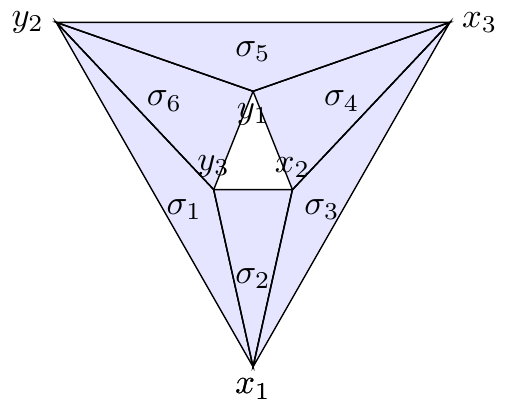}}
	\subfloat[$f(\Delta_1)$]{\includegraphics{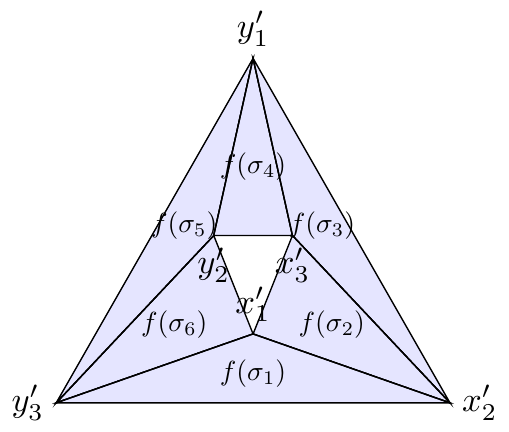}}
	\hspace{0.3cm}
	\subfloat[$\Gamma$]{\includegraphics[scale=1.4]{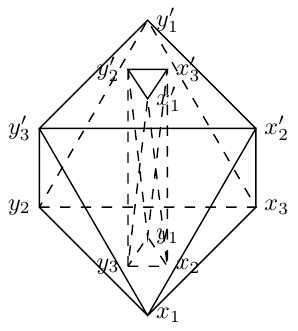}}
	\caption{The complexes $\Delta_1$, $f(\Delta_1)$ and $\Gamma$ when $d = 3$.}
\end{figure}

In the specific case of $d=3$, we have the triangulated manifold homeomorphic to $\Sp^2\times \Sp^0$, which consists of two disjoint spheres. In this case, the construction gives the boundary of two $3$-cross-polytopes, which is indeed a minimal triangulation. When $d=4$, one may check that the construction from above indeed triangulates $\Sp^2\times \Sp^1$. However, it has 16 vertices, so by the results in \cite{Klee-Novik-balanced LBT} and \cite{Zheng}, it is not a minimal triangulation. For $d\geq 5$, we check that $\Sigma$ satisfies all the conditions as described in Theorem \ref{thm: homology of product of spheres}.
\begin{lemma}\label{lm: simply connected}
	The complex $\Sigma$ given in Example \ref{main construction} is simply connected for $d\geq 5$.
\end{lemma}
\begin{proof}
	Since $B(d-3,d)$ contains the 2-skeleton of $\partial C_d^*$, it follows that both $\Delta_2$ and $f(\Delta_2)$ are simply connected. Then the underlying topological spaces of $A:=\Delta_2\cup N$ and $B:=N\cup f(\Delta_2)$ are also simply connected. It is easy to see that both $A\cup B$ and $A\cap B$ are path connected. The result follows since the union of two simply connected open subsets with path-connected intersection is simply connected.
\end{proof}
\begin{theorem}
	There exists a balanced triangulation of $\Sp^2\times \Sp^{d-3}$ with $4d$ vertices for $d\geq 3$.
\end{theorem}                                                               
\begin{proof}
	Our candidate is the complex $\Sigma$ in Example \ref{main construction}. It suffices to check the case $d\geq 5$. By Theorem \ref{thm: homology of product of spheres} and Lemma \ref{lm: simply connected}, it suffices to check that $\Sigma$ has the same homology as $\Sp^2\times \Sp^{d-3}$ for $d\geq 5$. Applying the Mayer-Vietoris sequence on the triple $(\Delta_2\cup f(\Delta_2), N, \Sigma)$,
	\[\cdots \to H_{i+1}(\Sigma)\to H_i(\partial \Delta_2\cup\partial f(\Delta_2))\to H_i(N)\oplus H_i(\Delta_2\cup f(\Delta_2))\to H_i(\Sigma)\to \cdots\]
	for all $i$. $H_i(\partial \Delta_2\cup\partial f(\Delta_2))=\Z\oplus \Z$ for $i=1, d-3, d-2$, and zero otherwise. Also by Lemma \ref{lemma: B(i,d) properties}, $H_i(\Delta_2\cup f(\Delta_2))\cong H_i(\Delta_2)\oplus H_i(\Delta_2)=\Z\oplus \Z$ for $i=d-3$, and zero otherwise. Consider the sequence
	\[0\to H_{d-2}(\Sigma)\to H_{d-3}(\partial \Delta_2\cup\partial f(\Delta_2))\stackrel{i^*}{\rightarrow} H_{d-3}(N)\oplus H_{d-3}(\Delta_2\cup f(\Delta_2))\to H_{d-3}(\Sigma)\to 0.\]
	Since $\|N\|\cong \|\partial \Delta_2\|\times [0,1]$ and the map $i^*$ (induced from the inclusion from $\partial \Delta_2\cup \partial f(\Delta_2)$ to $N$ and $\Delta_1\cup f(\Delta_2)$) is an injection, it simplifies to $0\to \Z\stackrel{i^*}{\rightarrow} \Z\oplus \Z\to H_{d-3}(\Sigma)\to 0$. Hence $H_{d-2}(\Sigma)=0$ and $H_{d-3}(\Sigma)=\Z$. Also
		\[0\to H_{d-1}(\Sigma)\to H_{d-2}(\partial \Delta_2\cup\partial f(\Delta_2))\to H_{d-2}(N)\oplus H_{d-2}(\Delta_2\cup f(\Delta_2))\to 0,\]
	which implies that $H_{d-1}(\Sigma)=\Z$.
	Finally, 
	\[0\to H_2(\Sigma)\to H_1(\partial \Delta_2\cup\partial f(\Delta_2))\stackrel{i'}{\rightarrow}H_1(N)\to H_1(\Sigma)\to 0.\]
	Again, the map $i'$ is surjective. Hence $H_2(\Sigma)=\Z$ and $H_1(\Sigma)=0$. It is clear to see that the other Betti numbers of $\Sigma$ are zero. Finally, the balancedness of $\Sigma$ follows from the construction.
\end{proof}

We list several properties of $\Sigma$. 
\begin{property}
	For $d \geq 4$, let $\Sigma$ be the balanced triangulation of $\Sp^2\times \Sp^{d-3}$ as constructed in Example \ref{main construction}.  Then, $\Sigma$ has the following face numbers.
	\begin{enumerate}
		\item $f_0(\Sigma)=4d$,
		\item $f_1(\Sigma)=4d(2d-3)$,
		\item $f_{d-1}(\Sigma)=(d+2)2^d-8d$.
	\end{enumerate}
\end{property}
\begin{proof}
	The complex $\Sigma$ has $4d$ vertices since $f_0(\Sigma)=f_0(\Delta_2)+f_0(f(\Delta_2))$. By the construction, there are $2d$ edges $e_1=\{x_2', y_3\},\dots, e_{2d}=\{x_1', y_2\}$ deleted from the cross-polytopes $\Gamma_1, \Gamma_2, \dots, \Gamma_{2d}$ to form $\Gamma$. Each $\Gamma_i$ and $\lk(e_i,\Gamma_i)$ are $(d-1)$-dimensional and $(d-3)$-dimensional octahedral spheres respectively, so we have that $f_1(\Gamma_i)=2d(d-1)$ and $f_1(\st(e_i,\Gamma_i))= 2(d-1)(d-2)+1$. Thus,
	\begin{align*}
	f_1(\Gamma_1\#\dots\# \Gamma_{2d})^\phi) &= \sum_{i=1}^{2d}\Big(f_1(\Gamma_i)- f_1(\st(e_i,\Gamma_i))-1\Big) \\
	&=4d^2(d-1)-4d(d-1)(d-2)-4d=4d(2d-3).
	\intertext{It follows from $f_1(\Delta_1)=f_1(\Delta_2)$ that $f_1(\Sigma)=f_1(\Gamma)=4d(2d-3)$. Similarly, since the facets in each $\st(e_i,\Gamma_i)$ are disjoint,}
	f_{d-1}(\Gamma)&=\sum_{i=1}^{2d} \Big(f_{d-1} (\Gamma_i) -2f_{d-1}(\st(e_i,\Gamma_i))\Big)\\&=2d(2^d-2^{d-1})=d2^d.
	\end{align*}
	It follows that $f_{d-1}(\Sigma)=f_{d-1}(\Gamma)-2f_{d-1}(\Delta_1)+2f_{d-1}(\Delta_2)=d2^d-4d+(2^{d+1}-4d)=(d+2)2^d-8d$.
\end{proof}


Here is another property concerning the automorphism group of $\Sigma$.

\begin{property}
	For $d \geq 4$, the complex $\Sigma$ admits a group action of $\Z_2\times \mathcal{D}_{2d}$.
\end{property}
\begin{proof}
	A simplicial map $g$ on the simplicial complex $\Sigma$ is an isomorphism if it gives a bijection on the facets of $\Sigma$.  A necessary condition for $g$ to be an automorphism is that it sends the missing edges in $\Sigma$ to missing edges in $\Delta$.  Define the following three permutations, modified from the proof of Theorem 1.2(b) in \cite{Klee-Novik-cs manifolds}:
	\begin{itemize}
		\item $D$ maps $x_j$ to $y_j$, $y_j$ to $x_j$, $x_j'$ to $y_j'$ and $y_j'$ to $x_j'$.
		\item $E'$ maps $x_j$ to $x_{d-j+1}'$, $y_j$ to $y_{d-j+1}'$, $x_j'$ to $x_{d-j+1}$ and $y_j'$ to $y_{d-j+1}$.
		\item $R'$ maps $x_d$ to $y_1$, $y_d$ to $x_1$, $x_j$ to $x_{j+1}$, $y_j$ to $y_{j+1}$, and similarly for $x_j'$ and $y_j'$.
	\end{itemize}
	The maps $D$ and $E'$ have order $2$, whereas $R'$ has order $2d$.  Also note that $E'$ is the permutation $E$ from \cite{Klee-Novik-cs manifolds} composed with a switching between the prime and nonprime vertices.  We know that of the edges in $\Sigma$, the only missing edges are edges between antipodal vertices in $\Gamma_i$ and the edges deleted when we join $\Gamma_i$ and $\Gamma_{i+1}$; they are $\{x_i'y_{i+1}\}, \{y_i'x_{i+1}\}$ for $1\leq i\leq d-1$ together with $\{x_d' x_1\}, \{y_d'y_1\}$.  It is straightforward to check that $D$, $E'$, and $R'$ are bijections on the vertices of $\Sigma$, and additionally fix setwise the set of missing edges. Since $E'R' = R'^{-1}E'^{-1}$, $E'$ and $R'$ generate $\mathcal{D}_{2d}$, and since $D$ commutes with both $E'$ and $R'$, we have that the three together generate $\Z_2 \times \mathcal{D}_{2d}$.
	
	By Theorem 1.2(b) of \cite{Klee-Novik-cs manifolds}, we have that facets in $\Delta_1$ and $f(\Delta_1)$, as well as those in $\Delta_2$ and $f(\Delta_2)$, are mapped bijectively by $g=D,E'$ or $R'$. Therefore, it suffices to show that the facets in the tubular neighborhood $N$ are mapped bijectively. Note that any facet $F$ in $N$ must also be contained in some $\Gamma_i$.  Therefore, the only way in which $g(F)$ could not be a facet of $\Sigma$ is if $g(F)$ is in the star of an edge which is deleted. However, as we observed above, $g$ gives bijection on the missing edges of $\Sigma$, i.e., $g(F)\in \st(g(e))$ for some missing edge $e$ if and only if $F\in \st(e)$. Hence $g$ is a bijection on the facets of $\Sigma$, and so $g\in \Aut(\Sigma)$.
	\end{proof}
\begin{remark}
	 We developed a Python/Sage program to produce our $4d$ vertex triangulation of $\Sp^2 \times \Sp^{d-3}$.  In addition, working with Lorenzo Venturello, we create a program to implement cross-flips on balanced simplicial complexes to attempt to reduce the number of vertices of a given triangulation. The program uses a simulated annealing approach, much like the method BISTELLAR uses. However, the complexity of finding shellable subcomplexes in the $d$-cross-polytope grows exponentially with $d$, and so the program is highly inefficient for $d > 4$. In addition, cross-flip sequences connecting two different triangulations (see \cite{Izmestiev-Klee-Novik} and \cite{Kubitzke-Venturello}) tend to be much more delicate and structured, and so simulated annealing works poorly on balanced complexes which cannot be immediately reduced by a cross-flip. So far, we have not found any balanced triangulation of $\Sp^2\times \Sp^{d-3}$ ($d\geq 5$) with less than $4d$ vertices.
\end{remark}
	Klee and Novik \cite{Klee-Novik-cs manifolds} showed that a balanced triangulation of a non-spherical $(d-1)$-manifold requires at least $3d$ vertices. It is not known that apart from the sphere bundle over the circle, if there are other $(d-1)$-manifolds that admit balanced triangulations with $3d$ vertices.  
	\begin{problem}
	Find a small balanced triangulation of $\Sp^i\times \Sp^j$.
	\end{problem}
	\begin{problem}
	Determine the sharp lower bound on the number of vertices required for a balanced triangulation of $\Sp^i\times \Sp^j$.
	\end{problem}

\section{Acknowledgements}
The authors are partially supported by the Mathematics REU program at the University of Michigan during summer 2018. We would like to thank Department of Mathematics at University of Michigan for providing strategic fund. We thank Lorenzo Venturello for supporting the computational aspect of this project. Our thanks also go to the two anonymous referees who provided invaluable comments and pointed out gaps in the proofs in an earlier version of this paper.
\appendix

\end{document}